\documentclass[a4paper,11pt, reqno]{amsart}
\usepackage[utf8]{inputenc}
\usepackage{amssymb,amsmath}
\usepackage{amscd}
\usepackage[all]{xy}
\usepackage[onehalfspacing]{setspace}
\usepackage{graphicx}
\usepackage{url}  
\usepackage[a4paper, left=2cm, right=2cm, top=3.5cm, bottom=3cm]{geometry}
\usepackage{tikz-cd}
\usepackage{array}
\usepackage{verbatim}
\usepackage{enumerate}

\def\P{\mathbb P}

\def\Z{\mathbb Z}

\def\AA{\mathcal A}
\def\BB{\mathcal B}
\def\FF{\mathcal F}
\def\GG{\mathcal G}

\def\NN{\mathcal N}
\def\OO{\mathcal O}

\def\RR{\mathcal R}
\def\SS{\mathcal S}
\def\TT{\mathcal T}

\def\CC{\mathbf C}
\def\DD{\mathbf D}
\def\EE{\mathbf E}
\def\LL{\mathbf L}
\def\VV{\mathbf V}
\def\QQ{\mathbf Q}
\def\WW{\mathbf W}
\def\XX{\mathbf X}
\def\YY{\mathbf Y}

\def\Hom{{\operatorname{Hom}}}

\def\RHom{{\operatorname{RHom}}}
\def\uRHom{\underline{\operatorname{RHom}}}

\def\Spec{{\operatorname{Spec}}}

\def\Cl{{\operatorname{Cl}}}

\def\Pic{{\operatorname{Pic}}}

\def\Bl{{\operatorname{Bl}}}

\def\Sing{{\operatorname{Sing}}}

\def\codim{{\operatorname{codim}}}

\def\rk{{\operatorname{rk}}}

\def\wt{\widetilde}
\def\wh{\widehat}
\def\ol{\overline}

\def\Db{\DD^b}

\def\Dperf{\DD^{\mathrm{perf}}}

\def\Kt{{\mathrm{K}}}

\theoremstyle{plain}
\newtheorem{dummy}{dummy}[section]

\newtheorem{theorem}[dummy]{Theorem}
\newtheorem{proposition}[dummy]{Proposition}
\newtheorem{lemma}[dummy]{Lemma}
\newtheorem{corollary}[dummy]{Corollary}
\newtheorem{example}[dummy]{Example}

\newtheorem{definition}[dummy]{Definition}
\newtheorem{remark}[dummy]{Remark}
\numberwithin{equation}{section}

\def\bal{\begin{aligned}}
\def\eal{\end{aligned}}

\newcolumntype{P}[1]{>{\centering\arraybackslash}p{#1}}
\newcolumntype{M}[1]{>{\centering\arraybackslash}m{#1}}

\title{Derived categories of nodal del Pezzo threefolds}
\author{Nebojsa Pavic and Evgeny Shinder}

\dedicatory{Dedicated to Yuri Gennadievich Prokhorov}

\address{Department of Mathematics and Scientific Computing, University of Graz,
Heinrichstra{\ss}e 36,
8020 Graz,
Austria}
\email{nebojsa.pavic@uni-graz.at}

\address{School of Mathematical and Physical Sciences, University of Sheffield,
Hounsfield Road, S3 7RH, UK}
\email{eugene.shinder@gmail.com}

\begin{document}

\maketitle

\begin{abstract}
    We give a complete answer for the existence of Kawamata type semiorthogonal decompositions of derived categories of nodal del Pezzo threefolds. 
    More precisely, we show that nodal del Pezzo threefolds of degree $1\leq d \leq 4$ have no Kawamata type decomposition and that all nodal del Pezzo threefolds of degree $5$ 
    admit a Kawamata decomposition. 
    For the proof we go through
    the classification of singular
    del Pezzo threefolds, compute
    divisor class groups of nodal
    del Pezzo
    threefolds of small degree and
    use projection from a line
    to construct Kawamata semiorthogonal
    decompositions for the degree $5$ case.
    An analogous decomposition of the nodal del Pezzo threefold of degree $6$ has been recently constructed by Kawamata.
    
    Our construction of the Kawamata decomposition
    for a singular del Pezzo threefold of degree $5$ fits into a family of semiorthogonal decompositions (which we call a relative tilting decomposition) interpolating between a Kawamata decomposition on a singular fiber and a full exceptional collection on the smooth fibers.
    
\end{abstract}

\tableofcontents

\section{Introduction}

Derived categories of coherent
sheaves, their equivalences and
semiorthogonal decompositions
provide a homological algebra
counterpart of the Minimal Model Program \cite{BO-preprint, BO-ci}. Namely, flops are expected to 
give rise to derived equivalences,
while flips, divisorial contractions
and Mori fiber spaces, as well as Sarkisov links
between them often correspond
to semiorthogonal decompositions.
Starting from smooth varieties,
singularities appear naturally and
can not be avoided
in the Minimal Model Program.
However, derived categories of coherent
sheaves are currently mostly understood in the 
smooth projective case
with the singular case providing both
technical and conceptual difficulties.
For instance, derived categories of smooth Fano threefolds
are well-understood, see \cite{Kuznetsov-rationality},
but the singular case is much less clear,
and constitutes an 
area of active current research.
Derived categories
of singular del Pezzo surfaces
have recently received a lot of attention
\cite{Kuznetsov-sextics, KKS, Xie-quintics},
and for singular Fano threefolds 
some sporadic nontrivial examples
have been constructed by Kawamata
\cite{Kawamata-CY, Kawamata-V6, Kawamata-P1113}.
Semiorthogonal decompositions of singular varieties
have been related to their smoothings for surfaces \cite{kawamata-smoothing} and Fano threefolds \cite{KS-fano}.

One notable conceptual difficulty of constructing
semiorthogonal decompositions of singular
varieties is appearance of obstructions
coming from algebraic K-theory, in the form
of the Brauer group \cite{KKS} or the $\Kt_{-1}$
group \cite{Kalck-Pavic-Shinder}, which do not appear in the smooth case. In some situations, such as for toric
surfaces a semiorthogonal decomposition
of a certain type of the derived category
can be constructed as soon as the $\Kt$-theoretic obstruction
vanishes \cite{KKS}.

In this paper we study semiorthogonal 
decompositions
of derived categories of nodal projective threefolds
in the framework of Kawamata semiorthogonal decompositions \cite{Kalck-Pavic-Shinder}, which are certain
generalizations of full exceptional collections
to a singular variety, see Definition \ref{def:Kawamata}.
We concentrate on nodal del Pezzo threefolds,
that is projective Fano threefolds $X$
with $-K_X = 2H$ for some $H \in \Pic(X)$, 
having ordinary double points.
Del Pezzo threefolds have
been classified by Iskovskikh, as part of the project of classifying smooth 
Fano threefolds
\cite{Iskovskikh} 
and by Fujita, as part of the project of
classifying possibly singular 
higher-dimensional del Pezzo varieties, see \cite[Chapter 3]{IskovskikhProkhorov} for an overview.
The degree $d := H^3$
of del Pezzo threefolds
satisfies $1 \le d \le 8$,
and nodal ones exist in degrees $1 \le d \le 6$.
Typical examples of del Pezzo threefolds
are
a cubic hypersurface in $\P^4$ ($d = 3$)
and a complete intersection of two quadrics in $\P^5$ ($d = 4$), see Theorem \ref{thm:delPezzo} for the complete classification.
The complexity of the geometry of del Pezzo
threefold $X$ depends
on its degree which
sits in one of three bands, cf. \cite{Prokhorov-G-Fano}:
$1 \le d \le 3$ (complicated: smooth 
$X$ is irrational,
but singular degenerations can be rational), 
$4 \le d \le 6$
(interesting and well-behaved: 
$X$ rational, singularities understood, cf. Corollary \ref{cor:singularities}), 
$d \ge 7$ (trivial: rational, smooth and rigid).
These bands roughly match
the complexity of the corresponding del Pezzo
surface hyperplane section.

We completely characterize nodal del Pezzo threefolds
having a Kawamata semiorthogonal decomposition
of their derived category.
Our main result is the following.

\begin{theorem}[see Corollary \ref{cor:main} for other equivalent statements]
Let $X$ be a nodal (non smooth) del Pezzo threefold.
The following condition are equivalent:
\begin{enumerate}
    \item $\Db(X)$ admits a Kawamata decomposition
    \item $\rk(\Cl(X)) = \rk(\Pic(X)) + |\Sing(X)|$
    \item $H^3 \in \{5,6\}$
\end{enumerate}
\end{theorem}

Condition (2) says that $X$ has as many Weil divisors
as the singularities allow for (see \eqref{eq:defect-seq}).
Theorem 1.1 provides a complete match between
obstructions to Kawamata decompositions coming from the $\Kt_{-1}$ group and Weil
divisors as developed in \cite{Kalck-Pavic-Shinder} and the possibility
of constructing such a decomposition when these obstructions
are trivial, which is analogous to the case of toric surfaces \cite{KKS}.
Regarding condition (3), 
we note that a smooth del Pezzo
threefold $X$ of degree $d$ has a full exceptional
collection if and only if $d \ge 5$ 
(equivalently when $h^{1,2}(X) = 0$, cf. \cite[Proposition 3.6]{Prokhorov-G-Fano}), and
as we exclude $d \ge 7$ cases which do not
have singular degenerations, Kawamata
decompositions in (1) can be considered
precisely as limits of exceptional collections
from the smooth del Pezzo threefolds,
see Remark \ref{rem:limit}.

Theorem 1.1 involves a mixture
of algebraic $\Kt$-theory,
computing class groups of nodal threefolds
and classification of Fano varieties.
The proof roughly goes 
as follows. The implication (1) $\implies$ (2)
is proved using vanishing of the
negative algebraic
$\Kt$-theory groups \cite{Kalck-Pavic-Shinder},
in (2) $\implies$ (3) we need the notion of 
defect of a linear system
for computing class groups \cite{cynk, Rams},
and in (3) $\implies$ (1) a construction of a semiorthogonal
decomposition based on the geometry of $X$ is required.
Existence of Kawamata decompositions for the nodal del Pezzo
threefold of degree $6$ goes back to the original work of Kawamata
\cite{Kawamata-V6}.
The main new ingredient in this paper
is existence of Kawamata decompositions
for nodal del Pezzo threefolds of degree $5$,
which was conjectured in \cite[Conjecture 1.3]{Kalck-Pavic-Shinder}.

Let us explain our approach to 
the derived category of
nodal del Pezzo threefolds $X$
of degree $5$. It is well known that
$X$ is a linear section of $\mathrm{Gr}(2,5) \subset \P^9$
and that
there are precisely three such threefolds up to isomorphism,
with
one, two or three nodes respectively, see e.g. \cite{Prokhorov-G-Fano}.
Following Iskovskikh \cite{Iskovskikh}
we project $X$ from a sufficiently general
line $L \subset X$. The resulting birational
map identifies the blow up of $X$ at $L$
with a blow up of a quadric threefold
$Q \subset \P^4$
along a nodal curve $C \subset Q$ 
of arithmetic genus zero
and degree three (see Theorem \ref{thm:degeneration}
for a more general statement which applies
to degenerations of del Pezzo threefolds of degree $4 \le d \le 6$).
Then the derived category of $X$ will be decomposed using mutations
of semiorthogonal decompositions coming from these blow ups.
The same approach for degree $4$ nodal 
del Pezzo threefolds
leads to a decomposition of their derived
category in terms of a copy of a nodal
associated curve of genus $2$,
see Theorem \ref{thm:decomp} for both $d = 4,5$
cases.
Moreover, using this method one can also recover Kawamata's decomposition of the derived category of the nodal $V_6$ as in \cite{Kawamata-V6}, see Remark \ref{rem:V6}.

Our construction works in any family of del Pezzo threefolds
of degree $5$ producing what we call a relative tilting decomposition which restricts to a full exceptional collection of objects on a smooth fibers and to a Kawamata semiorthogonal decomposition on the singular ones. In this sense our Kawamata semiorthogonal decomposition on a singular variety 
can be understood as a limit of a full exceptional collection from a smoothing,
see Remark \ref{rem:limit}.

We remark that in terms of the Minimal Model Program
the geometric input for our
semiorthogonal decompositions is provided by
the so-called 
Sarkisov links (for us these are simply
certain compositions of a blow up and a blow down) with non-smooth lci (nodal) centers,
and our point is that
from the derived categories perspective, 
producing nodal 
threefolds by blowing up and blowing down with
centers being nodal curves in the smooth locus
is completely analogous to using smooth blow ups.
It follows from Theorem \ref{thm:degeneration}
that all 
nodal del Pezzo threefolds of degrees $4$, $5$ and $6$
can be produced by this construction.
In general, it 
is not clear which nodal threefolds can be
obtained from smooth ones by such nodal blow ups and blow downs, which is a question
interesting in its own right, as well as in relation
to degenerations of intermediate Jacobians
and 
to rationality problems. 
Furthermore, it seems to us that existence
of Kawamata decompositions has to do with the
lci Sarkisov links as explained above, similarly
to how full exceptional collections are
supposed to correspond to varieties which are rational, or close to rational.
We hope to return to these questions
in the future.

\subsection*{Relation to other work.}
The main geometric input in this paper, 
the idea of blowing up singular curves to produce Gorenstein terminal threefolds, belongs to Yuri Prokhorov \cite{Prokhorov-G-Fano, 
Prokhorov-singular}.
The study of semiorthogonal decompositions of singular varieties
has been initiated by Kuznetsov and Kawamata,
see \cite{Kawamata-CY, Kuznetsov-sextics, KKS, Kawamata-V6, Xie-quintics, Kalck-Pavic-Shinder,
Kawamata-P1113, kawamata-smoothing}
for recent developments.
The case of nodal del Pezzo threefolds of degree $5$ has been
independently done by Fei Xie \cite{Xie} by analyzing small resolutions rather than performing 
nodal blow ups. The method
of \cite{Xie} is quite different to ours as
we do not resolve singularities but preserve
them in birational transformations.
A more general notion than Kawamata semiorthogonal decomposition
is the so-called categorical absorption \cite{KS-absorption}, \cite{KS-fano}, \cite{KKP}.

\subsection*{Acknowledgements}
E.S. was partially supported by the
EPSRC grant EP/T019379/1 Derived categories and algebraic $\Kt$-theory of singularities.
We thank Ivan Cheltsov, Martin Kalck, Alexander Kuznetsov, 
Yuri Prokhorov, Michael Wemyss,
Fei Xie for discussions and interest in our work.
We thank Alice Rizzardo and Theo 
Raedschelders for organizing
the 2019
Liverpool workshop ``The Geometry of Derived Categories'' where our
previous work \cite{Kalck-Pavic-Shinder}
has been presented
and this work
has been conceived.

\subsection*{Conventions and notation.}
We work over an algebraically
closed field $k$  of characteristic 
zero. 
Unless specified otherwise, all our varieties are quasiprojective.
By $\Db(X)$ we mean
the bounded derived category
of coherent sheaves on $X$
and $\Dperf(X)$ is the subcategory
of $\Db(X)$ consisting of perfect complexes.
More generally, we denote by $\Db(X, \RR)$ the
bounded derived category of coherent sheaves of
right $\RR$-modules,
for a locally free sheaf $\RR$ of algebras on $X$.
All functors such as pull-back $\pi^*$, pushforward $\pi_*$ and tensor product $\otimes$ when considered between derived categories are derived functors.

By letters $D$, $E$, $W$, $X$, $Y$ and similar in the standard font
we denote $k$-algebraic varieties.
By letters $\DD$, $\EE$, $\WW$, $\XX$, $\YY$ 
in typed in boldface we denote
proper and flat (but not always smooth) families over a smooth base $S$.
For $b \in S$, we write $\DD_b$, $\EE_b$ and so on for
the corresponding fiber.
By calligraphic
letters $\AA$, $\BB$, $\mathcal{C}$, $\TT$ 
we denote triangulated categories,
$\FF$, $\GG$ usually denote complexes of coherent
sheaves
and $\RR$ usually stands for a locally free
sheaf of algebras.
In the case when $\FF$
is a complex
of coherent sheaves on an $S$-variety $\XX$
we write $\FF_b$ for the derived
restriction of $\FF$ to $\XX_b$.

\section{Families of del Pezzo threefolds}

\subsection{Classification of del Pezzo threefolds}

By a del Pezzo threefold we mean a projective 
threefold 
$X$ with Gorenstein terminal singularities, 
such that $-K_X = 2H$,
for an ample line bundle
$H \in \Pic(X)$. 
By \cite[(1.1)]{Reid-minimal}, threefold Gorenstein terminal
singularities are precisely the
isolated compound du Val singularities.
A typical example are the $cA_n$ singularities
in which case
complete local rings $\wh{\OO}_{X,x}$, $x \in \Sing(X)$ are 
isomorphic
to hypersurface
singularities
\begin{equation}\label{eq:cDV}
k[[x,y,z,w]] / (x^2 + y^2 + g(z,w)) 
\end{equation}
with $g(z,w)$ a polynomial
having no repeated factors (that is $g(z,w)$
defines a reduced curve singularity).
We will pay particular attention to ordinary double
points, which is the case $g(z,w) = z^2 + w^2$,
and we call projective threefolds with isolated
ordinary double points nodal threefolds.

If $X$ is a del Pezzo threefold,
the integer $d = H^3$ is called the degree of $X$,
and we write $V_d$ for a del Pezzo threefold of degree $d$.

\begin{theorem}[Iskovskikh, Fujita] \cite[Theorem 3.3.1]{IskovskikhProkhorov}
\label{thm:delPezzo}
Del Pezzo threefolds are exactly the following:
\begin{itemize}
    \item $V_1 \subset \P(1,1,1,2,3)$ a degree $6$ hypersurface
    \item $V_2 \subset \P(1,1,1,1,2)$ a degree $4$ hypersurface (equivalently, a double cover of $\P^3$ branched in a quartic)
    \item $V_3 \subset \P^4$ cubic hypersurface
    \item $V_4 \subset \P^5$ intersection of two quadrics
    \item $V_5 \subset \P^6$ codimension three linear section of $\mathrm{Gr}(2,5) \subset \P^9$
    \item $V_6 \subset \P^7$ 
    hyperplane section of $\P^2 \times \P^2 \subset \P^8$
    \item $V_6' = (\P^1)^3$
    \item $V_7 = \Bl_p(\P^3)$, the blow up of a point  $p \in \P^3$
    \item $V_8 = \P^3$
\end{itemize}
\end{theorem}

\begin{remark}\label{rem:Vd-sing}
We note that $V_1$, $V_2$ 
can not pass through singular points of the ambient
weighted projective space; indeed, since 
our del Pezzo threefolds have only compound du Val
singularities hence
in particular
hypersurface singularities, 
the tangent space
at each point has dimension at most $4$,
hence 
the dimension of the tangent space
of the ambient space 
at this point should be at most $5$,
which is not the case for singular points of $\P(1,1,1,2,3)$
and $\P(1,1,1,1,2)$.
\end{remark}

By Bertini's Theorem, a 
general element $S \in |H|$
will be a smooth del Pezzo 
surface of degree $d$
which is convenient e.g. when analyzing
the Hilbert scheme
of lines on $V_d$.
The last three varieties $V_6', V_7, V_8$ are smooth,
and will not be of interest to us. The other types can
be nodal.
Singularities of del Pezzo threefolds,
in particular the maximal number of nodes
they can have, have
been studied by Prokhorov 
\cite{Prokhorov-G-Fano}
(see also Corollary \ref{cor:singularities}).

\subsection{Class groups of del Pezzo threefolds $V_1$, $V_2$, $V_3$}
We recall the relationship between Weil and Cartier
divisors on nodal threefolds.
For any nodal threefold $X$ we have an
exact sequence
\begin{equation}\label{eq:defect-seq}
0 \to \Pic(X) \to \Cl(X) \to \Z^{\Sing(X)}.
\end{equation}
Here the last map is given by 
restricting a Weil
divisor to the local class groups
of the singular points, see e.g. \cite[Corollary 3.8]{Kalck-Pavic-Shinder}.
In particular the quotient
$\Cl(X)/\Pic(X)$ is a free abelian group of finite rank
and
\[
\delta := \rk \; (\Cl(X) / \Pic(X))
\]
is called \emph{defect} of $X$.
It follows from \eqref{eq:defect-seq}
that $\delta \le |\Sing(X)|$;
if $\delta = |\Sing(X)|$ we say
that $X$ has \emph{maximal defect}.

We say that $X$ is \emph{maximally nonfactorial}
if the map $\Cl(X) \to \Z^{\Sing(X)}$ from 
\eqref{eq:defect-seq} is surjective.
It is clear that maximal nonfactoriality of $X$ 
implies that $X$ has maximal defect,
and the converse implication holds for Fano threefolds \cite[Proposition A.14]{KS-fano}.

\begin{remark}
The importance of the maximal nonfactoriality condition
is that it is a simple necessary condition
for existence of a Kawamata decomposition
for derived categories of nodal threefolds \cite{Kalck-Pavic-Shinder}, and it is also a necessary condition for a more general notion of categorical absorption \cite{KS-absorption}. 
Furthermore, we will show
that for nodal del Pezzo threefolds
this condition is also sufficient,
see Corollary \ref{cor:main}.
\end{remark}

We now use standard
computations of defect
to rule out del Pezzo threefolds of small degree
from being maximally nonfactorial.

\begin{proposition}[\cite{cynk,Rams}]\label{prop:cynk}
Let $W$ be a projective toric fourfold
with at most isolated quotient singularities,
and let $X \subset W$ be a (non smooth) nodal hypersurface
which does not intersect the singular locus of $W$.
Assume that $\OO_W(X)$ is ample
and that $\omega_W \otimes \OO_W(2X)$ is basepoint-free in the nonsingular locus of $W$.
Then $X$ does not have maximal defect.
\end{proposition}
\begin{proof}
Let $\mu$ be the number of singular points
of $X$, and $\delta$ be the defect.
Then by \cite[(4.2) and Corollary 4.2]{Rams} 
\[
\delta = \mu - (h^0(W, L) - h^0(W, L \otimes I_{\Sing(X)})),
\]
where $L = \omega_W \otimes \OO_W(2X)$.
Since $L$ is assumed to be basepoint free
on $W \setminus \Sing(W)$ and 
$\Sing(X) \subset W \setminus \Sing(W)$, we have
$h^0(W, L \otimes I_{\Sing(X)}) < h^0(W, L)$
so that $\delta < \mu$, and $X$ does not
have maximal defect.
\end{proof}

\begin{corollary}\label{cor:CynkRams}
Nodal (non smooth) $V_1$, $V_2$, $V_3$ do not have
maximal defect.
\end{corollary}
\begin{proof}
By Theorem \ref{thm:delPezzo}, in each case $X = V_d$ is a hypersurface in 
a weighted projective space:
\begin{center}
\begin{tabular}{|c|c|c|c|c|}
    \hline
     $d$ & $W$ & $\OO_W(X)$ & $\omega_W$ & $\omega_W \otimes \OO_W(2 X)$\\
    \hline
     $1$ & $\P(1,1,1,2,3)$ & $\OO(6)$ & $\OO(-8)$ & $\OO(4)$ \\ 
    \hline
     $2$ & $\P(1,1,1,1,2)$ & $\OO(4)$ & $\OO(-6)$ & $\OO(2)$ \\ 
    \hline
     $3$ & $\P^4$ & $\OO(3)$ & $\OO(-5)$ & $\OO(1)$ \\ 
    \hline
\end{tabular}
\end{center}
Here $\OO_W(X)$ is ample and $\omega_W \otimes \OO_W(2X)$ is basepoint free on the nonsingular
locus of $W$ (the only basepoint is $[0:0:0:0:1]$
in $d = 1$ case).
By Remark \ref{rem:Vd-sing}, $X$ does not pass through
the singular points of $W$.
All the conditions of Proposition
\ref{prop:cynk} are verified and the result follows.
\end{proof}

\begin{remark}
It follows from the proof of Proposition \ref{prop:cynk}
that $|\Sing(X)| - \delta$ equals
the number of linearly independent conditions
that the points from $\Sing(X)$ impose
on forms of degree $e$ on the corresponding
weighted projective space, where
$e = 4$ for $V_1$, $e = 2$ for $V_2$
and $e = 1$ for $V_3$.
\end{remark}

Note that this 
approach of computing
the defect does not apply to nodal
intersections of two singular
quadrics $V_4$.
We will see in Corollary
\ref{cor:main} that nodal 
$V_4$ never have maximal defect,
while nodal 
$V_5$, $V_6$ are always maximally nonfactorial
(hence have maximal defect).
In the case of $V_4$, its defect
can be read off from the associated
curve of arithmetic genus two, see
Theorem \ref{thm:degeneration}
and proof of Corollary \ref{cor:main}.

\subsection{Geometry of 
del Pezzo threefolds $V_4$, $V_5$, $V_6$}

The key construction in the geometry of del Pezzo
threefolds
is projection from a line $L \subset X$ which has been
used by Iskovskikh \cite{Iskovskikh}
in the classification in the smooth case.

\begin{definition}
Let $V$ be a del Pezzo threefold.
We call a line $L \subset V$
a \emph{standard line} if the following conditions
hold
\begin{enumerate}
    \item[(a)] $L$ is contained in the smooth locus of $V$
    \item[(b)] $L$ is not contained in any plane $\Pi \subset V$
    \item[(c)] $\NN_{L/V} \simeq \OO_L^2$
\end{enumerate}
\end{definition}
 
 For (b), see \cite[3.7, 3.7.1]{Prokhorov-G-Fano} for planes on del Pezzo threefolds;
 for degree $d \ge 3$ a plane is simply a $\P^2$ embedded
 linearly in the ambient space. 

The following example shows that for nodal del Pezzo threefolds the scheme of lines
is in general reducible.  
Reducibility of the 
Fano variety of lines 
on singular cubic threefolds in relation to defect 
have been
recently studied 
\cite{Marquand-Viktorova}.

\begin{example}
Let $V$ be a $3$-nodal
del Pezzo threefold of degree $5$. Then $V$ has a small
resolution $\wh{V}$ which is isomorphic to $\Bl_{3}(\P^3)$,
the blow up of $\P^3$ in $3$ general points $P_1, P_2, P_3$: the morphism
$\wh{V} \to V$ is given by the linear
system $|2H-P_1-P_2-P_3|$.
It
contracts three lines passing through pairs
of 
points $P_i, P_j$ \cite[7.1, 7.4]{Prokhorov-G-Fano}.
Lines on $V$ correspond to smooth
rational curves $C \subset \wh{V}$
with the property $C \cdot (2H-E_1-E_2-E_3) = 1$ where
$E_i \subset \wh{V}$ are exceptional divisors
of the blow up.
These curves
can be of the following types:
\begin{itemize}
    \item proper preimages of lines in $\P^3$ passing through
    exactly one of the points $P_i$ 
    \item proper preimages of conics in $\P^3$ passing through
    $P_1, P_2, P_3$
    \item lines on one of the exceptional divisor
    $\P^2 \simeq E_i \subset \wh{V}$
\end{itemize}
General lines of the first type are standard
and the lines of the second
and the third type are contained in a plane (see \cite[7.2]{Prokhorov-G-Fano} for the list
of planes on $V$),
so they are not standard.
\end{example}

We present projection from a line construction
for a degeneration of 
del Pezzo threefolds:

\begin{definition}\label{def:dP-fibration}
A del Pezzo 
threefold fibration 
is a flat proper morphism 
$f: \VV \to S$ 
with smooth $\VV$ and $S$,
such that each fiber $\VV_b$ is a del Pezzo threefold
with at most Gorenstein terminal singularities,
and such that
there exists $H \in \Pic(\VV)$
which restricts to ample generator of $\Pic(\VV_b)$ for each $b \in S$.
\end{definition}

We call a subscheme $\LL \subset \VV$ a \emph{standard
family of lines}
if $\LL \to S$ is a $\P^1$-bundle and
for all $b \in S$
 $\LL_b \subset \VV_b$ is a standard line.

\begin{lemma}\label{lem:standardline1}
(i) Any del Pezzo threefold $V$ of degree $3 \le d \le 6$ with terminal
Gorenstein singularities 
contains a standard line.

(ii) A del Pezzo threefold $V$ as in (i)
with a standard line
$L \subset V$
can be included as a central fiber
into a del Pezzo fibration
$\VV \to S$ admitting a standard family of lines $\LL \to S$
with $\LL_0 = L$.
\end{lemma}
\begin{proof}
(i) Let us show that there exists a line $L \subset V$
in the smooth locus, and not contained in any plane $\Pi \subset V$.
Taking a general hyperplane section of $V$ not passing through
singular points of $V$ gives rise to a smooth del Pezzo
surface $S \subset V$. 
Let $i: S \to V$ be the embedding
morphism; since $S$ lies in the smooth locus of $V$,
we have a restriction homomorphism $i^*: \Cl(V) \to \Pic(S)$.
If $\Pi \subset V$ is a plane, and $L \subset S \cap \Pi$, 
then $L = i^*(\Pi)$ by degree reasons.
It follows from \cite[Corollary 3.9.2]{Prokhorov-G-Fano} that
$i^*$ is
not surjective, and since $\Pic(S)$ is generated
by lines $L \subset S$, there exists a line
not contained in any plane $\Pi \subset V$.

We have a short exact sequence of normal bundles
\[
0 \to \OO_L(-1) \to \NN_{L/V} \to \OO_L(1) \to 0. 
\]
Since vector bundles on $L \simeq \P^1$ split into a direct sum of line
bundles, $\NN_{L/V}$ is either $\OO \oplus \OO$
or $\OO(-1) \oplus \OO(1)$.
In both cases $h^0(\NN_{L/V}) = 2$, $h^1(\NN_{L/V}) = 0$
so that $L$ corresponds to a smooth point of a two-dimensional
irreducible component $T \subset F(V)$, where $F(V)$ is the Hilbert scheme of lines on $V$.

We claim that lines parametrized by 
$t \in T$ cover $V$.
Indeed, otherwise they would cover
a surface on $V$, 
and the only integral surface
containing a two-parameter family
of lines is a plane (because
a general point on this surface will be
its vertex).
This is impossible
since by construction
$L$ does not lie on a plane $\Pi \subset V$.

We now compute the normal bundle for a general line
following a method of Iskovskikh.
Let $P$
be the universal line over $T$
so
that we have a diagram
\[
\xymatrix{
T & P \ar[l]_p \ar[r]^\phi & V \\
}
\]
where $\phi$ is generically finite and $p$ is a $\P^1$-bundle.
We claim that $\phi$ is etale
on $p^{-1}(T \setminus D)$
where $D \subset T$ is a divisor.
Let $T^\circ$ (resp. $P^\circ$)
be the smooth locus of $T$ (resp. $S$).
We have $K_{P^\circ} = -2H + R$,
where $R$ is a divisor
supported at the ramification locus of $\phi$.
On the other hand, by the formula
for canonical class of a projective
bundle
$K_{P^\circ} = p^*(K_{T^\circ}(\det(\EE))) - 2H$,
and comparing the two expressions
for $K_{P^\circ}$ we deduce
that $R$ is supported over 
$p^{-1}(D)$, for some divisor $D$.
After enlarging $D$ 
to include $\Sing(T)$,
we obtain
that $\phi$ is etale away
from $p^{-1}(D)$.

Now a general line $L$ parametrized
by $t \in T$ does not
pass through singular points of $V$,
and $\phi$ is etale over $L$.
It follows that
\[
\NN_{L/V} \simeq \NN_{p^{-1}(t)/S} = \OO_L \oplus \OO_L,
\]
which finishes the proof.

(ii) Consider a general four-dimensional
del Pezzo fourfold $\ol{\VV}$ of degree $d$ containing $V$
(see Theorem \ref{thm:delPezzo}).
A general hyperplane section
$V'$ of $\ol{\VV}$
passing through $L$ will be smooth
and not containing any
singular points of $V$. 
Hence blowing up the base locus
of a pencil of such hyperplane sections
we obtain a flat proper morphism
$\VV := \Bl_{V \cap V'}(\ol{\VV}) \to \P^1$
with a smooth total space,
and removing singular
fibers other than $V$ will give the required degeneration.
\end{proof}

\begin{theorem}
\label{thm:degeneration}
Let $\VV \to S$ be a del Pezzo threefold
fibration
of degree $4 \le d \le 6$.
Let $\LL \subset \VV$ be a standard family
of lines.
Let $\YY$ be the blow up $\YY = \Bl_{\LL}(\VV)$ 
and $\EE \subset \YY$
be the exceptional divisor.

Then the line bundle $\OO(H - \EE)$ is relatively globally generated.
Let $\pi: \YY \to \WW$ be the induced surjective morphism
and let
$\QQ = \pi(\EE)$.
Then $\pi|_\QQ: \QQ \to S$
is a smooth two-dimensional
quadric fibration contained in the smooth locus of $\pi$.
Furthermore, there exists a
smooth subscheme $\CC \subset \QQ$ flat
and proper over $S$ 
such that $\pi: \YY \to \WW$ is the blow up
of $\WW$ along $\CC$.
The possibilities for $\WW$ and $\CC$
are given in the table:

\smallskip

\begin{center}
\begin{tabular}{|c|c|c|c|}
    \hline
    $d$    & $\WW \to S$ & bidegree of $\CC_b \subset \QQ_b$ & description of $\CC_b$ \\
    \hline
     $4$ &  $\P^3$-fibration & $(2,3)$ & arithmetic genus two curve \\
     \hline
     $5$ &  $Q^3$-fibration & $(1,2)$ & generalized twisted cubic \\
     \hline
     $6$ &  $\P^1 \times \P^2$-fibration & $(1,1)$ & conic \\
    \hline
     $6$ &  $\P^1 \times \P^2$-fibration & $(0,2)$ & two disjoint lines  \\
    \hline
\end{tabular}
\end{center}
\smallskip

The $Q^3$-fibration $\WW \to S$
in the case $d = 5$ has smooth or nodal fibers.
If $\VV_b$ is smooth, then 
the curve $\CC_b$ and the threefold $\WW_b$ are also smooth.
In the degree $d = 6$ case, in the notation
of Theorem~\ref{thm:delPezzo},
blowing up the $(1,1)$ curve
case corresponds to $V_6$ and blowing up the
$(0,2)$ curve
corresponds to $V_6'$.
\end{theorem}

The birational maps in Theorem~\ref{thm:degeneration}
are summarized in the following diagram
\begin{equation}\label{diag:bir-model-proj-line}
\xymatrix{
 & \EE \ar@{^{(}->}[r]^{i} \ar@{->}[ld]_p & \YY \ar@{<-^{)}}[r]^{j} \ar@{->}[ld]_{\sigma} \ar@{->}[rd]^{\pi} & \DD \ar@{->}[rd]^q & \\
 \LL \ar[drr] \ar@{^{(}->}[r] & \VV \ar[dr] &   & \WW \ar@{<-^{)}}[r] \ar[dl] & \CC \ar[dll] \\    
 & & S & &\\
}
\end{equation}

\begin{proof}
By generic smoothness, general fibers of $\VV \to S$ are smooth.
To simplify the notation let us assume that the only possibly 
singular fiber is $\VV_0$ over the point $0 \in S$.
Fiberwise the morphism $\pi$ resolves projection from the line
$\LL_b \subset \VV_b$ so that $H - \EE$ is base point free on $\YY$.
Furthermore, fibers of $\pi$ are residual components obtained
by intersecting $\VV_b$
with planes passing through $\LL_b$. 
Since $d \ge 4$, each del Pezzo threefold
$\VV_b$
is an intersection of quadrics \cite[Theorem 3.2.4(iii)]{IskovskikhProkhorov}, 
so that fibers of $\pi$ can be empty, consist
of a reduced point or isomorphic to $\P^1$. In other words for each $b \in S$, $\pi$ contracts
secant lines to $\LL_b \subset \VV_b$, and induces a birational
morphism $\YY_b \to \WW_b$. 
We have
$h^0(\YY_b, H-\EE_b) = h^0(\VV_b, I_{\LL_b}(H)) = d$
and using \cite[Lemma 2.2.14]{IskovskikhProkhorov} for smooth fibers
$\WW_b$
    \[
    \deg(\WW_b) = (H-\EE_b)^3 = H^3 - (3\cdot H + K_{\VV_b}) \cdot \LL_b + 2g - 2 = d - 3
    \]
so that $\WW_b$ is a subvariety of degree $d-3$ in $\P^{d-1}$; since $(H-\EE_b)^3$ is independent of $b$, 
the same is true for the singular fibers as well. 
As $\deg(\WW_b) = \codim(\WW_b) + 1$,
it is a so-called \emph{variety of minimal degree} and
there are only
the following possibilities for $\WW_b$ \cite[Theorem 2.2.11]{IskovskikhProkhorov}:
\begin{itemize}
    \item $d=4$: $\WW_b = \P^3$
    \item $d=5$: $\WW_b \subset \P^4$ 
    is a quadric (possibly singular)
    \item $d=6$: $\WW_b = \P^1 \times \P^2 \subset \P^5$ (the Segre embedding),
    a cone over a cubic scroll or
    a cone over a rational twisted cubic curve.
\end{itemize}

The linear system $|H - \EE|$ restricts to each $\EE_b \simeq \P^1 \times \P^1$ as $\OO(1,1)$, so that $\pi|_\EE$ is an isomorphism
onto a smooth quadric fibration $\QQ \subset \WW$.
Thus we have an isomorphism 
\[
\VV \setminus \LL \simeq \YY \setminus \EE \simeq \WW \setminus \QQ
\]
and in particular $\WW \setminus \QQ$ is smooth, 
each $\WW_b \setminus \QQ_b$ is smooth and
$\WW_0 \setminus \QQ_0$
has at most terminal Gorenstein singularities.
This rules out all types of singular $\WW_b$ except for a nodal quadric
threefolds
in the $d = 5$ case.
It follows that
$\QQ_b$ is contained in the smooth locus
of $\WW_b$ and since $\WW \setminus \QQ$
is also smooth, $\WW$ is a smooth fourfold.

By construction $\pi$ contracts a divisor
onto a relative curve $\CC \subset \QQ$.
Thus $\pi: \YY \to \WW$
is a birational
projective morphism of smooth
varieties, with at most one-dimensional
fibers, hence by 
Danilov's decomposition theorem
\cite{Danilov},
$\pi$ is a composition of blow ups
with smooth centers.
Since all fibers of $\pi$ are irreducible,
in fact $\pi$ is a blow up of $\CC \subset \WW$ and $\CC$ is smooth.

It remains to check the type
of the curve $\CC_b \subset \QQ_b$.
Since $\CC_b$ is a flat family,
it suffices to consider $b \ne 0$.
    The normal bundle of $\QQ_b$ in $\WW_b$ is
    \begin{itemize}
        \item $d = 4$: $\NN = \OO(2,2)$
        \item $d = 5$: $\NN = \OO(1,1)$
        \item $d = 6$: $\NN = \OO(0,1)$.
    \end{itemize}
    Computing the normal bundle
    of $\EE_b$ in $\YY_b$ in 
    two ways using the blow ups $\sigma$ and
    $\pi$, we obtain $\NN(-\CC_b) \simeq \OO(0,-1)$.
    It follows that
    bidegree and the type of $\CC_b$ must be as claimed.
    
    In the $d = 6$ case the varieties
    $V_6$ and $V_6'$ can be distinguished
    using their Picard rank, which then
    distinguishes the curve types:
    the $(0,2)$ curve consists
    of two connected components,
    hence the blow up increases 
    Picard rank by two, while
    the $(1,1)$ curve is connected,
    hence the blow up increases the Picard rank
    by one.
\end{proof}

Let us consider the restriction of the diagram \eqref{diag:bir-model-proj-line} to fibers over $b \in S$.
Let us write $V = \VV_b$, $Y = \YY_b$ and so on.
Considering the local charts 
of a blow up $\pi_b$,
we see that if the complete local 
equation for $C$ on a quadric
surface is $g(z,w) = 0$,
then the blow up $Y$ has
an equation $xy + g(z,w) = 0$,
and so $Y$, and hence $V$ 
will have Gorenstein
terminal (in fact, $cA_n$ \eqref{eq:cDV}) singularities if
and only if $C$ is reduced.
Furthermore, we see that $V$ is at most 
nodal
if and only if both $W$ and
$C$ are at most nodal,
in 
which case
\begin{equation}\label{eq:nodes}
|\Sing(V)| = |\Sing(Y)| = 
|\Sing(W)| + |\Sing(C)|.
\end{equation}
We call $C$ an \emph{associated curve}
to $V$ (this curve
depends on $V$ and a chosen standard line $L \subset V$).

\begin{corollary}\label{cor:singularities}
Gorenstein terminal del Pezzo threefolds
of degrees listed below 
can only 
have the following types of singularities:
\begin{itemize}
    \item $d = 4$: only $cA_n$ singularities;
    at most six nodes
    \item $d = 5$: only nodal; at most three nodes
    \item $d = 6$: only nodal; at most one node for $V_6$ and $V_6'$ is smooth
\end{itemize}
\end{corollary}
\begin{proof}
Indeed, only
these types of singularities
can occur when blowing up the curve $C$
of the given type. The number of nodes
can be computed 
using \eqref{eq:nodes}.
\end{proof}

From diagram \eqref{diag:bir-model-proj-line}
we obtain the relations between
divisor classes in $\Pic(\YY)/\Pic(S)$, which we will use later:
\begin{itemize}
\item $d = 4$:
\begin{equation}\label{eq:relations4}
h:= H - \EE, \quad E = 2h - \DD , \quad H  = 3h - \DD \end{equation}

\item $d = 5$:
\begin{equation}\label{eq:relations5}
h:= H - \EE, \quad \EE = h - \DD , \quad H  = 2h - \DD \end{equation}
\end{itemize}

\section{Derived categories of  del Pezzo fibrations}

\subsection{Relative tilting semiorthogonal decompositions}

Let $\TT$ be a $k$-linear
triangulated category.

\begin{definition}[\cite{BO-preprint}, \cite{bondal-kapranov}]
A collection $\AA_1,\ldots ,\AA_m$ of full triangulated subcategories of $\TT$ is called a
semiorthogonal 
decomposition of $\TT$,
if 
\begin{itemize}
    \item for all $1\leq i < j\leq m$,
    $\Hom(\AA_j, \AA_i) = 0$,
\item the smallest triangulated subcategory of $\TT$ containing $\AA_1 ,\ldots ,\AA_m$ coincides with $\TT$.
\end{itemize}
\end{definition}

If each $\AA_i$ is admissible in $\TT$ that is,  the inclusion functor $\AA_i \subset \TT$
has both adjoint functors,
then we say that the semiorthogonal decomposition is admissible.

We write $\TT=\langle\AA_1 , \ldots , \AA_m\rangle$
for a semiorthogonal decomposition; all decompositions
we consider will be admissible, and in many cases admissibility is automatic.
The following result is a typical
starting
point for constructing
admissible semiorthogonal decompositions.

\begin{theorem}[Orlov \cite{orlov-monoidal}]\label{thm:blowup}
Let $Z \subset X$ be smooth subvariety
of pure codimension $c$.
Let $\pi\colon\wt X\to X$ be the blow up
of $X$ with center $Z$,
and let $i: E \to \wt{X}$ be the exceptional
divisor, with projective
bundle structure $p: E \to Z$.
Then there is an admissible semiorthogonal
decomposition 
\[
\Db(\wt{X}) = \langle 
\Db(Z)_{-(c-1)}, \dots, \Db(Z)_{-1}, 
\pi^*(\Db(X)) \rangle ,
\]
where $\Db(Z)_{-j}$
is the image of the fully faithful
functor $(i_* p^*(-)) \otimes \OO(jE): \Db(Z) \to \Db(\wt{X}).$
\end{theorem}

We introduce the type of
semiorthogonal
decompositions of singular varieties that 
we are interested in.
\begin{definition}\cite{Kalck-Pavic-Shinder}\label{def:Kawamata}
Let $X$ be a Gorenstein projective variety.
An admissible semiorthogonal decomposition
\[
\Db(X) = \langle \BB ,  \AA_1, \dots, \AA_r 
\rangle
\]
is called a {\it Kawamata decomposition}
if $\BB \subset \Dperf(X)$ and each $\AA_i$ 
is equivalent
to $\Db(R_i)$, with $R_i$ a finite-dimensional $k$-algebra.
\end{definition}

For example, a smooth projective
variety with a full exceptional
collection admits a Kawamata decomposition
with each $\AA_i \simeq \Db(k)$ (and $\BB = 0$),
and Kawamata decompositions can be considered
as a way of generalizing exceptional
collections to singular varieties.
For examples of curves admitting 
Kawamata semiorthogonal decompositions see Remark \ref{rem:burban-algebras} below.
For other examples see \cite{KKS}, \cite{Kalck-Pavic-Shinder}, \cite{Kawamata-CY}, \cite{Kawamata-V6}, \cite{Kawamata-P1113}, \cite{kawamata-smoothing}.

We will now introduce a relative version of a Kawamata semiorthogonal decomposition.
For that we use semiorthogonal decompositions over a smooth base $S$ defined by Kuznetsov \cite{Kuznetsov-base} and 
reformulated by Perry in a more abstract setting of stable $\infty$-categories 
\cite{Perry-relative}.

An $S$-linear category $\TT$ is an appropriately enhanced $k$-linear triangulated category endowed with an action
\[
\otimes\colon \Db(S) \times \TT \to \TT
\]
and it automatically admits relative Hom-objects \cite[2.3.1]{Perry-relative}
\[
\uRHom_{\TT/S}(-, -) \colon \TT^\mathrm{op} \times \TT \to \Db(S)
\]
with a functorial isomorphism for all 
$a \in \Db(S)$, $t, u \in \TT$
\[
\RHom_S(a, \uRHom_{\TT/S}(t, u)) = \RHom_\TT(a \otimes t, u).
\]

For example, if $\TT = \Db(\XX)$ for a scheme $f\colon \XX \to S$, $\TT$ is $S$-linear with $\otimes$ defined by $f^*(-) \otimes -$ and with
$
\uRHom_{\TT/S}(-,-) = f_*\uRHom_\XX(-,-).
$

More generally, if $\RR$ is a locally free sheaf of $\OO_S$-algebras, we consider
the derived category of coherent sheaves of right $\RR$-modules
$\Db(S, \RR)$, see \cite[2.1]{Kuznetsov-quadric-fibration}.
The category
$\TT = \Db(S, \RR)$ is $S$-linear
via the natural tensor product functor
\[
\otimes: \Db(S) \times \Db(S, \RR) \to \Db(S, \RR)
\]
and $\uRHom_{\TT/S}(-,-) = f_*\uRHom_{\RR}(-, -)$.
Note that $\Db(S) = \Db(S, \OO)$.

\medskip

For the rest of this subsection we assume that
$f\colon \XX \to S$ is a flat projective 
morphism between smooth varieties.

\begin{definition}
By a relative tilting semiorthogonal decomposition we mean an $S$-linear 
semiorthogonal decomposition
\[
\Db(\XX) = \langle \AA_1, \ldots, \AA_r \rangle
\]
such that we have $S$-linear equivalences
$\AA_i \simeq \Db(S, \RR_i)$ 
for some locally free sheaves of algebras $\RR_i$ on $S$.
\end{definition}

The name relative tilting decomposition 
is due to the fact that the image $T_i \in \Db(\XX)$
of each $\RR_i$
satisfies $f_*{\uRHom}(T_i, T_i) = \RR_i[0]$
so when $S = \Spec(k)$ we get a so-called
pretilting object \cite{kawamata-smoothing} which is tilting (i.e. generating) if $r = 1$.

By the standard properties of base change we obtain
for every $b \in S$ an induced Kawamata semiorthogonal decomposition
\[
\Db(\XX_b) \simeq \langle \Db(\RR_1 \otimes k(b)), \ldots , \Db(\RR_r \otimes k(b)) \rangle.
\]

Before we give examples of relative tilting decompositions we state a simple admissibility property.
For an $S$-linear category $\TT$ a relative Serre functor 
\cite[4.6]{Perry-relative}
is an $S$-linear equivalence
$\SS_{\TT/S}: \TT \to \TT$ 
with a functorial isomorphism for $t, u \in \TT$
\[
\uRHom_{\TT/S}(t, \SS_{\TT/S}(u)) \simeq
\uRHom_{\TT/S}(u, t)^\vee.
\]
For example for a morphism $\XX \to S$, the $S$-linear category
$\TT = \Db(\XX)$
 admits a relative Serre functor given by
    $- \otimes \omega_{\XX/S}[\dim(X)-\dim(S)]$.
    The following result is well-known in the absolute case.

\begin{lemma}\label{lem:admissibility}
Given an $S$-linear semiorthogonal decomposition
$\TT = \langle \AA_1, \ldots \AA_m
\rangle$ 
if both $\TT$ and all the components $\AA_i$
have a relative Serre functor, then 
the decomposition 
is admissible.
\end{lemma}

\begin{proof}
It is clear that $\AA_1$ is left admissible.
Using the standard argument involving the
Serre functor for $\TT$ and $\AA_i$
we can write
\[
\TT = \langle \AA_2, \ldots, \AA_m, \SS^{-1}_{\TT/S}\AA_1 \rangle
\]
which implies that $\AA_2$ is left admissible.
By induction we obtain that all $\AA_i$ are left admissible.
A similar argument shows right admissibility.
\end{proof}

Our first example of a relative
tilting decomposition comes from quadric fibrations
\cite{Kuznetsov-quadric-fibration}.

\begin{theorem}[{\cite[Theorem 4.2]{Kuznetsov-quadric-fibration}}]\label{thm:quadratic-fibrations}
Let $f\colon \XX \to S$ be a flat
quadric fibration with smooth $\XX$ and $S$.
Then there is a relative
tilting semiorthogonal
decomposition
\[
\Db(\XX) \simeq \langle \Db ( S , \BB_0 ) , f^\ast\Db(S) \otimes \OO_X (1) , \ldots , f^\ast\Db(S) \otimes \OO_X(n-2) \rangle ,
\]
where 
$\BB_0$ is the sheaf of even parts of Clifford algebras on $S$.
\end{theorem}

Another example of tilting decomposition is given by
degenerations of rational curves.

\begin{proposition}\label{prop:nodal-curves-degen}
Let $f: \XX \to S$ be a flat projective morphism with general members isomorphic a smooth
rational curve, and singular fibers having nodal singularities. 
Then, after passing to an open covering of $S$, we have an $S$-linear semiorthogonal decomposition
\begin{equation}
\label{eq:Db-curves-rel}
\Db(\XX) = \langle \Db(S, \RR), f^*(\Db(S)) \rangle.
\end{equation}
\end{proposition}

Recall that if $X$ is a projective
nodal curve of arithmetic genus zero
and components $X_1, \ldots, X_m$ we have
$\Pic(X) = \Z^m$ generated by isomorphism
classes of line bundles $L_i$ such that
\begin{equation}\label{eq:def-Li}
L_i|_{X_j} \simeq \OO_{X_j}(-\delta_{ij}).
\end{equation}
Consider a line bundle on $X$ given by
\begin{equation}\label{eq:def-L}
L = \bigoplus_{i=1}^m L_i^{\oplus r_i}
\end{equation}
with all $r_i > 0$.
We will call any such line bundle $L$ a \emph{minimally negative bundle}.
For example if $X = \P^1$, then minimally negative bundles have the form $\OO(-1)^{\oplus r}$, $r > 0$.

\begin{lemma}\label{lem:T-curves}
Assume that $f\colon \XX \to S$ is as in Proposition \ref{prop:nodal-curves-degen}.
Assume that $T$ is a locally free sheaf on $\XX$ such that for every $b \in B$
the restriction $T|_{\XX_b}$ 
is a minimally negative line bundle. Then $\RR := f_*{\uRHom(T, T)}$ is a locally free sheaf of algebras on $S$.
Define a functor
\begin{equation}
\label{eq:PhiT}
\Phi_T: \Db(S, \RR) \to \Db(\XX), \quad
\FF \mapsto f^*(\FF) \otimes_{f^*(\RR)} T.
\end{equation}
Then we have an $S$-linear semiorthogonal decomposition
\begin{equation}
\label{eq:Db-curves-rel-S}
\Db(\XX) = \langle \Phi_T (\Db(S, \RR)), f^*(\Db(S)) \rangle.
\end{equation}
\end{lemma}
\begin{proof}
This is a standard argument in the spirit of \cite{Samokhin}.
The fact that $\RR$ is a locally sheaf boils down to a computation on each fiber which uses \eqref{eq:def-Li}.
It is then a standard computation that $\Phi_T$
is fully faithful, its image is semiorthogonal to $f^*(\Db(S))$ and that
these two subcategories form an $S$-linear semiorthogonal decomposition of $\Db(\XX)$.
\end{proof}

\begin{proof}[Proof of Proposition \ref{prop:nodal-curves-degen}]
For every point $b \in S$
there is an open neighborhood $U \subset S$ of $b$
and a relative hyperplane section $D \subset \XX_U$,
which is finite 
over $U$, and intersects every 
irreducible
component
of every fiber.
Furthermore we can assume that these intersections are transverse.
To simplify the notation let us assume that $S = U$.
Let $d$ be the degree of $D$ over $S$.

Then $\XX' = \XX \times_S D$ is a 
degree $d$ cover $p\colon \XX' \to \XX$
which
admits 
a section $i: D \to \XX'$.
Let $T := \pi_*(\OO(-D))$. 
By construction it restricts to a minimally negative vector bundle on the fibers of $f$.
We apply Lemma \ref{lem:T-curves} to $T$ to get the result.
\end{proof}

\begin{remark}\label{rem:burban-algebras}
Restricting the tilting decomposition \eqref{eq:Db-curves-rel}
to each genus zero curve $\XX_b$, 
for $b \in S$ 
we obtain a semiorthogonal decomposition
\[
\Db(\XX_b) = \langle \Db(R), \OO \rangle 
\]
with a finite-dimensional algebra $R$.
This algebra is 
isomorphic to a matrix algebra
when $\XX_b \simeq \P^1$.
On the other hand, if $\XX_b$
is a nodal curve, then
$R$ is the algebra defined by Burban \cite{burban}.
It can be described explicitly as path algebra of a quiver, see \cite[Theorem 2.1]{burban}
or \cite[Remark 4.14]{Kalck-Pavic-Shinder} for uniform
formulations.
For example, if $\XX_b$ is a chain of two smooth projective
lines, that is a nodal conic, 
then $R$ is the path algebra
of the quiver
\begin{equation}\label{eq:Burban1}
\bal
\xymatrix{
1 \ar@/^/[r]^a & \ar@/^/[l]^{a^*} 2}&, \;\;\;\; 
 a a^* = a^* a= 0.
\eal
\end{equation}
This algebra is isomorphic to the Clifford algebra of a nodal quadric of odd dimension (cf \cite[Example 5.6]{Kawamata-CY}).

Similarly, for
a nodal 
chain of three smooth
projective lines,
$R$ is the path algebra
of the quiver
\begin{equation}\label{eq:Burban2}
\bal
\xymatrix{
1 \ar@/^/[r]^a & \ar@/^/[r]^{b}\ar@/^/[l]^{a^*} 2
& \ar@/^/[l]^{b^*} 3}&, \;\;\;\; 
 a a^* = a^* a= b b^* = b^*b = 0 .
\eal
\end{equation}

We call the path
algebra for
\eqref{eq:Burban1}
the 
\emph{single Burban algebra}
and
the path algebra
for \eqref{eq:Burban2}
the \emph{double Burban} algebra.
Both of these algebras will
appear in the semiorthogonal
decompositions of nodal del Pezzo
threefolds.
\end{remark}

\subsection{Main results}
The following theorem 
generalizes corresponding
results in the smooth case
to degenerations:
semiorthogonal decomposition
for smooth $V_4$
constructed by Bondal and Orlov \cite{BO-ci}
and an exceptional collection 
for smooth for $V_5$ constructed
by Orlov \cite{orlov-V5};
see \cite[2.4]{Kuznetsov-rationality}
for a uniform treatment of smooth Fano threefolds.

\begin{theorem}\label{thm:decomp}
Let $\VV \to S$ be a nodal
del Pezzo fibration 
of degree $d \in \{4, 5\}$
with a standard family of lines $\LL \subset \VV$.
Let $\CC \to S$ be the associated family of 
curves 
as in Theorem \ref{thm:degeneration}, corresponding
to $\LL$.

\begin{itemize}
    \item[(i)] If $d = 4$, then $\CC\to S$ 
    is a family of at most nodal
    curves of arithmetic genus $2$ and
    we have an admissible $S$-linear semiorthogonal decomposition
    \[
    \Db(\VV) = \langle \Db(\CC), \Db(S) , \Db(S)(H) \rangle.
    \]

\item[(ii)] If $d = 5$, 
then $\CC\to S$ is a family of generalized twisted cubics
and, after passing to an open cover of $S$, 
there is an admissible $S$-tilting semiorthogonal decomposition
\[
\Db(\VV) = \langle \Db(S, \RR_1), \Db(S, \RR_2), \Db(S) , \Db(S)(H) \rangle 
\]
where $\RR_1$ is the sheaf of Burban algebras
corresponding to $\CC \to S$ 
and $\RR_2$ is the sheaf of Clifford algebras
corresponding 
to the quadric threefold fibration $\WW \to S$.
\end{itemize}
\end{theorem}

Here the embedding of $\Db(S)$ and $\Db(S)(H)$
is obtained by the pullback with respect to $f\colon \VV \to S$
 and twisting by $\OO(H)$ in the latter case
and the embedding of other categories are specified in the proof of the theorem.
As a consequence, we obtain:

\begin{corollary}\label{cor:decomp}
For a 
nodal del Pezzo threefold 
$V$ of degree 
$d \in \{4, 5\}$
we have the following admissible
semiorthogonal decompositions.
\begin{itemize}
\item[(i)] If $d = 4$, then 
\[
\Db(V) = \langle \Db(C), \OO , \OO(H) \rangle
\]
where $C$ is the associated nodal curve of arithmetic genus two.

\item[(ii)] If $d = 5$, there is a Kawamata type semiorthogonal decompositions
\[
\Db(V) = \langle \Db(R_1), \Db(R_2), \OO, \OO(H) \rangle
\]
where $R_1$ and $R_2$ are Burban algebras.
\end{itemize}

\end{corollary}

\begin{proof}[Proof of Corollary \ref{cor:decomp}]
By Lemma \ref{lem:standardline1} (ii) there is a nodal 
del Pezzo fibration  $f\colon \VV \to S$ with a standard family of lines $\LL\subset \VV$, where $S$ 
is a smooth affine curve, $0 \in S$ a point and $\VV_0 = V$. 
By base change of the decomposition
constructed in Theorem \ref{thm:decomp} 
to $0 \in S$, we get that 
\[
\Db( V ) = \langle \AA_V, \OO , \OO(H) \rangle ,
\]
where $\AA_V\simeq \Db(C)$ and $\AA_V \simeq \langle \AA_C , \AA_Q \rangle $ for $d = 4$ and $d = 5$ respectively.
\end{proof}

In both cases $d = 4$ and $d = 5$
there is an induced semiorthogonal decomposition
of $\Dperf(V)$, see \cite[Proposition 1.10 and 1.11]{Orlov-LG}
and \cite[Theorem 4.4]{Kalck-Pavic-Shinder}

Before we prove Theorem \ref{thm:decomp}
we use it to deduce
a complete 
structural result about
derived categories of nodal del Pezzo threefolds,
which implies Theorem 1.1 in the 
Introduction.

\begin{corollary}\label{cor:main}
Let $V$ be a nodal (non smooth)
del Pezzo threefold
of arbitrary degree $d \ge 1$
and $\AA_V = \langle \OO, \OO(H)\rangle^\perp$ be the main component of the derived category $\Db(V)$.
The following conditions are equivalent:
\begin{enumerate}
    \item $\Db(V)$ admits a Kawamata decomposition
    \item $\Db(V)$ has an admissible decomposition
    with all components equivalent to derived
    categories of finite-dimensional algebras
    \item $\AA_V$ has an admissible decomposition
    with all components equivalent to derived
    categories of finite-dimensional algebras
    \item $V$ is maximally nonfactorial
    \item $V$ has maximal defect
    \item $d \in \{5,6\}$
\end{enumerate}
\end{corollary}
\begin{proof}
We prove the following chain of implications 
\[
(3) \implies 
(2) \implies (1) \implies (4) \implies (5)
\implies (6)
\implies (3). 
\]
The first two 
implications are trivial.

(1) $\implies$ (4) is \cite[Theorem 1.1]{Kalck-Pavic-Shinder}.

(4) $\implies$ (5) is trivial.

(5) $\implies$ (6)
We first note that by Corollary \ref{cor:CynkRams}
we have $d \ge 4$. 
In the $d = 4$ 
case
we can relate the defect to the negative K-group $\Kt_{-1}(V)$
\cite{Kalck-Pavic-Shinder} as follows.
By Theorem \ref{thm:decomp}
(or using
the diagram \eqref{diag:bir-model-proj-line} directly)
we see that $\Kt_{-1}(V) = \Kt_{-1}(C) \ne 0$
(see \cite[Corollary 3.3]{Kalck-Pavic-Shinder} for $\Kt_{-1}$ of a nodal curve),
hence $V$ does not have
maximal defect by \cite[Proposition 3.5]{Kalck-Pavic-Shinder}.
Finally if $d \ge 7$, then $V$ can not be singular
by Theorem \ref{thm:delPezzo}.

(6) $\implies$ (3) follows from Corollary \ref{cor:decomp} for the $d = 5$
case.
For the $d = 6$ case
the result is 
\cite[7.2]{Kawamata-V6} (the variety $X$
in \cite[7.2]{Kawamata-V6} is the unique
nodal $V_6$ by \cite[Theorem 7.1]{Prokhorov-G-Fano});
see also Remark \ref{rem:V6}.
\end{proof}

\subsection{Proof of Theorem \ref{thm:decomp}}
The proof relies on projection from a standard {family of lines}
from Theorem \ref{thm:degeneration}.
More concretely, we describe {$\Db(\YY)$} in terms of the blow up {$\sigma\colon \YY\to \VV$} in Proposition \ref{prop:Y-to-V_d} and in Proposition \ref{prop:Y-to-W_d} we describe {$\Db(\YY)$} in terms
of {$\pi\colon \YY\to \WW$} for cases $d= 4,5$ separately. 
The proof of Theorem \ref{thm:decomp} follows then readily by comparing these two descriptions of {$\Db(\YY)$}.

Let {$f_{\EE}:\EE \to S$} be the exceptional divisor of {$\sigma\colon \YY\to \VV$, where $i\colon \EE \hookrightarrow \YY$} denotes the corresponding embedding and let {$f_{\DD}\colon \DD\to S$} be the exceptional divisor of {$\pi\colon \YY\to \WW$ with inclusion morphism $j\colon \DD \hookrightarrow \WW$} as in {Theorem \ref{thm:degeneration}}.

\begin{lemma}\label{lem:blowup-mutation}

\noindent{(i)} We have {the following} equality of subcategories
in {$\Db(\YY)$}:
\begin{equation}\label{eq:blowup-mutation-E-H}
 { \langle f_{\YY}^\ast\Db(S) (- \EE ) , f_{\YY}^\ast\Db(S)  \rangle = \langle f_{\YY}^\ast\Db(S) , i_\ast f_{\EE}^\ast\Db(S) \rangle
  = \langle i_\ast f_{\EE}^\ast\Db(S)  ,f_{\YY}^\ast\Db(S)(- \EE ) \rangle  }
\end{equation}
and
\begin{equation}\label{eq:blowup-mutation-D-h}
 {   \langle f^\ast_{\YY}\Db(S) (-h) ,f^\ast_{\YY}\Db(S) (\DD-h)\rangle  = \langle j_\ast f^\ast_{\DD}\Db(S) (\DD-h) , f^\ast_{\YY}\Db(S)(-h) \rangle . }
\end{equation}
(ii) We have {$\Hom ( f^\ast_{\YY}\Db(S) (H- \EE ), i_\ast f^\ast_{\EE}\Db(S)(2H) [k] )=0$} for all $k$.
 \end{lemma}

\begin{proof}
(i) The pair {$i_\ast f^\ast_{\EE}\Db(S) (\EE ), f_{\YY}^\ast\Db( S)$} is {semiorthogonal} (this is part of
Theorem \ref{thm:blowup} with $j = 1$)
and \eqref{eq:blowup-mutation-E-H} is obtained
by standard mutations using the distinguished triangle
\[
{\OO_{\YY}(- \EE ) \to \OO_{\YY} \to \OO_{\EE}.}
\]
The same argument proves \eqref{eq:blowup-mutation-D-h}.

(ii) By adjunction it suffices to prove vanishing of
{$\Hom(\Db(S) , \Db(S) \otimes {f_{\EE}}_\ast \OO_{\EE}(H+ \EE ) [k] )$}
and this follows from
\[
{{f_{\EE}}_\ast  (\OO_{\EE}(H+ \EE )) = {f_{\LL}}_\ast(p_*(\OO_\EE(\EE))(H)) = 0.}
\]
\end{proof}
In what follows we work with the main
components 
of derived categories,
in the following cases 
\begin{itemize}
    \item $\CC$ a family of
    of rational curves: $\AA_\CC = \langle \OO \rangle^\perp$,
    \item $\QQ$ a quadric fibration of relative dimension three: $\AA_{\QQ} = \langle
\OO(-1) , \OO, \OO(1)
\rangle^\perp$,
    \item $\VV$ a del Pezzo fibration:
    $\AA_\VV = \langle \OO, \OO(H) \rangle^\perp$.
\end{itemize}

When mutating semiorthogonal
decompositions in the proofs below,
we use
a slight abuse of notation
by not specifying the embedding
functor of $\AA_{{\CC}}$, $\AA_{{\QQ}}$ and $\AA_{{\VV}}$
in the derived category of $\Db({\YY})$
when this embedding is clear from context
or not important for us.

\begin{proposition}\label{prop:Y-to-V_d}
Let $d\in \{ 4, 5, 6\}$ and let {$f_{\VV}\colon \VV \to S$ be a degeneration of del Pezzo threefolds containing a standard family of lines $\LL \subset \VV$}. Let {$f_{\YY}\colon \YY \to S$} be the blow up of {$\VV$} along {$\LL$}. There is an {$S$-linear}
admissible semiorthogonal decomposition
{\[
\Db(\YY) = \langle \AA_{\VV} , f^\ast_{\YY}\Db(S) ( \EE -H) , f^\ast_{\YY}\Db(S) (- \EE ) , f^\ast_{\YY}\Db(S) , f^\ast_{\YY}\Db(S) (H- \EE ) \rangle . 
\]
}
\end{proposition}

\begin{proof}
We will use
the blow up formula (Theorem \ref{thm:blowup}) together with mutations from Lemma \ref{lem:blowup-mutation}.

Theorem \ref{thm:blowup} applied to the blow up {$\sigma\colon\YY\to \VV$ along $\LL$} gives a semiorthogonal decomposition
\begin{align}\label{eq:first-step}
    \Db(\YY) &=\langle \Db(\LL )_{-1}, \sigma^*\Db(\VV )\rangle \nonumber \\
    &=\langle ( f_{\LL}^\ast\Db(S) \otimes \OO_{\LL}(-2))_{-1} , ( f^\ast_{\LL}\Db(S ) \otimes \OO_{\LL}(-1) )_{-1}, \sigma^*\Db(\VV )\rangle \nonumber \\
    &=\langle i_\ast f_{\EE}^\ast \Db(S)(\EE-2H), i_\ast f^\ast_{\EE}\Db(S)(\EE-H), \sigma^*\Db(\VV)\rangle \nonumber \\
    &=\langle \sigma^*\Db(\VV), i_\ast f_{\EE}^\ast \Db(S)(\EE-2H-K_{\YY}), i_\ast f^\ast_{\EE}\Db(S)(\EE -H- K_{\YY} ) \rangle. 
\end{align} 

Here we used {Orlov's projective bundle formula \cite[Theorem 2.6]{orlov-monoidal} in the second equation,} that {$p^\ast ( \OO_{\LL}(-1) )\simeq \OO_{\EE} (-H)$ (as $H\cdot \LL_b =1$ on each fiber of $\VV_b$, $b \in S$)} in the {third} equality and Serre duality in the {fourth} equality.

We plug in 
{$\Db(\VV)=\langle f_{\VV}^\ast\Db(S)(-H) , \AA_{\VV} , f_{\VV}^\ast\Db(S) \rangle $} and the canonical class{
\begin{equation}\label{eq:canon-div-blow-up}
K_{\YY} \equiv -2H + \EE \pmod{\Pic(S)}
\end{equation} 
}
in \eqref{eq:first-step} to obtain {
\begin{equation}
\Db(\YY)=\langle f_{\YY}^\ast\Db(S)(-H) , \sigma^\ast \AA_{\VV} , f_{\YY}^\ast\Db(S) , i_\ast f_{\EE}^\ast \Db(S), i_\ast f^\ast_{\EE}\Db(S)(H)  \rangle .
\end{equation}
}
Using Serre duality together with \eqref{eq:blowup-mutation-E-H} of Lemma \ref{lem:blowup-mutation}, we perform
the following sequence of mutations: {
\begin{align}\label{eq:second-step}
\Db(\YY) &= \langle f_{\YY}^\ast\Db(S)(-H) , \sigma^\ast \AA_{\VV} , f_{\YY}^\ast\Db(S) , i_\ast f_{\EE}^\ast \Db(S), i_\ast f^\ast_{\EE}\Db(S)(H)  \rangle \nonumber \\
 &=\langle  \sigma^\ast \AA_{\VV} , f_{\YY}^\ast\Db(S) , i_\ast f_{\EE}^\ast \Db(S), i_\ast f^\ast_{\EE}\Db(S)(H) ,  f_{\YY}^\ast\Db(S)(H - \EE ) \rangle  \nonumber \\
    &= \langle  \sigma^\ast \AA_{\VV} , f_{\YY}^\ast\Db(S)(-\EE) , f_{\YY}^\ast\Db(S) ,  f_{\YY}^\ast\Db(S)(H - \EE ) , f^\ast_{\YY}\Db(S)(H) \rangle \nonumber \\
    &= \langle f^\ast_{\YY}\Db(S)(\EE - H) ,  \sigma^\ast \AA_{\VV} , f_{\YY}^\ast\Db(S)(-\EE) ,   f_{\YY}^\ast\Db(S) ,  f_{\YY}^\ast\Db(S)(H - \EE )   \rangle \nonumber \\
    &= \langle  \AA_{\VV} , f^\ast_{\YY}\Db(S)(\EE - H) ,   f_{\YY}^\ast\Db(S)(-\EE) ,   f_{\YY}^\ast\Db(S) ,  f_{\YY}^\ast\Db(S)(H - \EE )   \rangle
     .
\end{align}
}

 {In the fourth equality we used Serre duality again 
 and we performed a left} mutation of {$\sigma^\ast \AA_{\VV}$} through {$f^\ast_{\YY}\Db(S)( \EE - H )$ in the fifth equality.} 
 This concludes the proof of the proposition.
\end{proof}
On the other hand, we can also decompose {$\Db(\YY )$} semiorthogonally with respect to the blow up {$\pi\colon\YY \to\WW$}.

\begin{proposition}\label{prop:Y-to-W_d}
Let $d\in \{ 4, 5 \}$ 
and consider diagram \eqref{diag:bir-model-proj-line}.
\begin{enumerate}[(i)]
\item \label{item:Y-to-W_d-i} Let $d=4$. There is an {$S$-linear}
admissible semiorthogonal decomposition  
\[
\Db( \YY ) = \langle \Db ( \CC ) , f^\ast_{\YY} \Db(S) (-h) , f^\ast_{\YY} \Db(S) ( \DD - 2h) , f^\ast_{\YY} \Db(S) , f^\ast_{\YY} \Db(S) (h)\rangle .
\]

\item \label{item:Y-to-W_d-ii} Let $d=5$ so 
that {$\WW = \QQ^3$ is a quadric threefold fibration with at most nodal quadrics as fibers}.
There is an {$S$-linear}
admissible semiorthogonal decomposition
{
\[
\Db(\YY ) = \langle \AA_{\CC} , \AA_{\QQ^3} , f^\ast_{\YY} \Db(S) (-h) , f^\ast_{\YY} \Db(S) ( \DD - h) , f^\ast_{\YY} \Db(S) , f^\ast_{\YY} \Db(S)(h) \rangle .  
\] 
}
\end{enumerate}
\end{proposition}

\begin{proof}
Theorem \ref{thm:blowup} applied to the blow up {$\pi\colon\YY \to \WW$} gives a semiorthogonal decomposition {
\begin{equation}\label{eqn:decomp-Y-to-W_d}
\Db(\YY ) = \langle \Db(\CC )_{-1} , \Db(\WW ) \rangle .
\end{equation}
}
We consider the cases $d = 4, 5$ separately.

Case $d=4$: Here, {$\WW \to S$} is a Zariski locally trivial $\P^3$-bundle with a relative hyperplane class $h$.
We can write out \eqref{eqn:decomp-Y-to-W_d} as {
\begin{align}\label{eq:decomp_d_4}
\Db(\YY) & = \langle \Db(\CC)_{-1} , f^\ast_{\YY} \Db(S) (-h) , f^\ast_{\YY} \Db(S) , f^\ast_{\YY} \Db(S)(h) , f^\ast_{\YY} \Db(S)(2h) \rangle \nonumber \\
& = \langle f^\ast_{\YY} \Db(S)( \DD - 2h) , \Db(\CC)_{-1} , f^\ast_{\YY} \Db(S) (-h) , f^\ast_{\YY} \Db(S) , f^\ast_{\YY} \Db(S)(h)  \rangle \nonumber \\
& = \langle \Db(\CC) , f^\ast_{\YY} \Db(S)( \DD - 2h) , f^\ast_{\YY} \Db(S) (-h) , f^\ast_{\YY} \Db(S) , f^\ast_{\YY} \Db(S)(h)  \rangle , 
\end{align}
}
where we used {the projective bundle formula \cite[Theorem 2.6]{orlov-monoidal} in the first equation}, Serre duality 
with {$K_{\YY} = -4h + \DD + f_{\YY}^\ast K_S$ in the second equality}
and left mutation of $\Db(\CC)$ through {$f^\ast_{\YY} \Db(S) (\DD -2h)$} in the {third} equality.

Furthermore, using the relations \eqref{eq:relations4} we can rewrite the pair {$f^\ast_{\YY} \Db(S) ( \DD -2h) , f^\ast_{\YY} \Db(S) (-h)$ as $f^\ast_{\YY} \Db(S) (- \EE ) ,f^\ast_{\YY} \Db(S) ( \EE -H)$}.
But from the decomposition of Proposition \ref{prop:Y-to-V_d}, we see that this pair is completely orthogonal.
Thus we can swap {$f^\ast_{\YY} \Db(S) ( \DD -2h)$ and $f^\ast_{\YY} \Db(S) (-h)$} in \eqref{eq:decomp_d_4} and we obtain item \eqref{item:Y-to-W_d-i}.
Admissibility {and $S$-linearity} of the constructed
semiorthogonal
decomposition follows 
as in the proof of Proposition \ref{prop:Y-to-V_d}.

Case $d=5$: We have {$\WW = \QQ^3$}, where {$\QQ^3$} is a quadric {fibration} in {$\P_S^4$} and the associated curve {$\CC \to S$ is a family of curves of arithmetic genus zero}. 
We rewrite \eqref{eqn:decomp-Y-to-W_d} as
\begin{align*}
\Db( \YY ) & = \langle \AA_{\CC} , i_* f^\ast_{\DD} \Db(S)(-h + \DD ), \AA_{\QQ^3}, f^\ast_{\YY} \Db(S)(-h) , f^\ast_{\YY} \Db(S) , f^\ast_{\YY} \Db(S)(h) \rangle \\
& = \langle \AA_{\CC} , \AA_{\QQ^3} ,  i_\ast f^\ast_{\DD} \Db(S)(-h + \DD ), f^\ast_{\YY} \Db(S)(-h) , f^\ast_{\YY} \Db(S) , f^\ast_{\YY} \Db(S)(h) \rangle \\
& = \langle \AA_{\CC} , \AA_{\QQ^3} , f^\ast_{\YY} \Db(S)(-h) , f^\ast_{\YY} \Db(S)(-h + \DD ) , f^\ast_{\YY} \Db(S) , f^\ast_{\YY} \Db(S)(h) \rangle.
\end{align*}
Here we used {\cite[Theorem 4.2]{Kuznetsov-quadric-fibration} in the first equation,} left mutation of {$\AA_{\QQ^3}$} through {$ i_\ast f^\ast_{\DD} \Db(S)(-h + \DD )$} in the second equation and \eqref{eq:blowup-mutation-D-h} in the third equation. 
Admissibility of the components
is automatic by Lemma \ref{lem:admissibility}.
\end{proof}

\begin{proof}[Proof of Theorem \ref{thm:decomp}] 
By Proposition \ref{prop:Y-to-V_d} we obtain
\[
\Db(\YY) =  \left\{
\begin{array}{rc}
     \langle \AA_{\VV}, f^\ast_{\YY} \Db(S)(-h) ,f^\ast_{\YY} \Db(S)( \DD -2h) , f^\ast_{\YY} \Db(S) , f^\ast_{\YY} \Db(S)(h)\rangle , & \text{if $d = 4$} \\
      \langle \AA_{\VV}, f^\ast_{\YY} \Db(S) (-h) , f^\ast_{\YY} \Db(S)( \DD -h) , f^\ast_{\YY} \Db(S) , f^\ast_{\YY} \Db(S)(h)\rangle , &  \text{if $d = 5$}\\
\end{array}
\right.
\]
where we used \eqref{eq:relations4}, \eqref{eq:relations5},
to rewrite {$\EE$} and $H$
in terms of divisors pulled back from {$\WW$}.
Comparing these expressions with Proposition \ref{prop:Y-to-W_d}, we obtain that {
\[
\AA_{\VV} \simeq  \left\{
\begin{array}{rc}
     \Db( \CC ) , & \text{if $d = 4$} \\
     \langle \AA_{\CC} , \AA_{\QQ^3} \rangle , &  \text{if $d = 5$}\\
\end{array}
\right.
\]
}
The obtained semiorthogonal decompositions are admissible by Lemma \ref{lem:admissibility}
because all the categories involved have relative Serre functors.
\end{proof}

\begin{remark}\label{rem:comparison}
Both cases $d = 4, 5$ of {Corollary \ref{cor:decomp}}
fit into Kuznetsov's framework of homological projective
duality, however his interpretation of the
category $\AA_{V}$ is different.
Specifically, in \cite[6.5]{Kuznetsov-HPD1},
for $d = 4$,
$\AA_{V}$ is shown to be equivalent to a derived category
of modules over an algebra over a different nodal curve,
while in our approach no sheaves of algebras is needed.
Using homological projective
duality 
one can decompose 
$\AA_{V}$ for $d = 5$
into derived categories of finite dimensional
dg-algebras (see
\cite[6.1]{Kuznetsov-HPD1} for the smooth case), 
which is less restrictive than the Kawamata
decomposition we construct for $d = 5$.
\end{remark}

\begin{remark}\label{rem:limit}
In the $d = 5$ case the two exceptional
objects generating $\AA_{V}$ 
in the smooth case degenerate to derived categories
of Burban's algebras \eqref{eq:Burban1},
\eqref{eq:Burban2}. Let us explain this in detail.
For a $1$-nodal
degeneration one exceptional object
degenerates to a single Burban algebra.
For a $2$-nodal degeneration one exceptional
object degenerates to a double Burban algebra,
or both exceptional objects degenerate to
single Burban algebras.
Finally for a $3$-nodal degeneration
two exceptional objects degenerate
to single and double Burban algebras respectively.
\end{remark}

\begin{remark}\label{rem:V6}
The construction of a Kawamata
decomposition for the nodal $V_6$ 
\cite[7.2]{Kawamata-V6}
can also be done using the projection
from a line as for $d = 4, 5$ cases.
Indeed, for $d = 6$
we can apply the blow up formula (Theorem \ref{thm:blowup}) to the blow up $\pi\colon Y \to W$, where $W = \P^2\times \P^1$ as in \eqref{diag:bir-model-proj-line} and the
associated curve $C$
is of type $(1,1)$. 
After mutating some line bundles, we can bring $\Db(Y)$ into the same form as in Proposition \ref{prop:Y-to-V_d} and we can express $\AA_{V}$ as
\[
\AA_{V} \simeq \langle \mathcal{E}_1 , \mathcal{E}_2 , \mathcal{E}_3 , \AA_C \rangle ,
\]
where $\mathcal{E}_i$ are exceptional objects.
\end{remark}


\end{document}